\documentclass[10pt]{article}
\font\smallit=cmti10

\usepackage{amsthm,amsfonts,amssymb,amsmath,epsf, verbatim}
\usepackage{url}
\usepackage[verbose]{placeins}
\usepackage[letterpaper,top=2cm,bottom=2cm,left=3cm,right=3cm,marginparwidth=1.75cm]{geometry}
\usepackage[sort,numbers]{natbib} 
\usepackage{caption}
\usepackage{graphicx}

\makeatletter

\renewcommand\section{\@startsection {section}{1}{\z@}
{-30pt \@plus -1ex \@minus -.2ex}
{2.3ex \@plus.2ex}
{\normalfont\normalsize\bfseries\boldmath}}

\renewcommand\subsection{\@startsection{subsection}{2}{\z@}
{-3.25ex\@plus -1ex \@minus -.2ex}
{1.5ex \@plus .2ex}
{\normalfont\normalsize\bfseries\boldmath}}

\renewcommand{\@seccntformat}[1]{\csname the#1\endcsname. }

\makeatother

\newtheorem{theorem}{Theorem}
\newtheorem{lemma}{Lemma}
\newtheorem{proposition}{Proposition}
\newtheorem{corollary}{Corollary}

\theoremstyle{definition}
\newtheorem{definition}{Definition}
\newtheorem{conjecture}{Conjecture}
\newtheorem{remark}{Remark}

\begin{document}

\begin{center}
\uppercase{\bf \boldmath Complete characterization of $2$-near perfect numbers with 2 distinct prime factors}
\vskip 20pt
{\bf Richard Fearon}\\
{\smallit University of Notre Dame, Notre Dame, Indiana, USA}\\
{\tt rfearon@nd.edu}\\
\vskip 10pt
{\bf Henry Foushee}\\
{\smallit Yale University, New Haven, Connecticut, USA}\\
{\tt henry.foushee@yale.edu}\\
\vskip 10pt
{\bf Benjamin Porosoff}\\ 
{\smallit California Institute of Technology, Pasadena, California, USA}\\
{\tt bporosof@caltech.edu}\\
\vskip 10pt
{\bf Alexander Skula}\\
{\smallit Massachusetts Institute of Technology, Cambridge, Massachusetts, USA}\\
{\tt skula@mit.edu}\\
\vskip 10pt
{\bf Joshua Zelinsky}\\
{\smallit Department of Mathematics, Hopkins School, New Haven, Connecticut, USA}\\
{\tt jzelinsky@hopkins.edu}\\
\vskip 10pt
{\bf Kyle Zhang}\\
{\smallit Massachusetts Institute of Technology, Cambridge, Massachusetts, USA}\\
{\tt kylez@mit.edu}\\
\end{center}
\vskip 20pt
\centerline{\smallit Received: , Revised: , Accepted: , Published: } 
\vskip 30pt

\centerline{\bf Abstract}
\noindent
A positive integer $n$ is \emph{$2$-near perfect} if $\sigma(n)=2n+d_1+d_2$ for two distinct positive divisors $d_1,d_2$ of $n$. We give a complete classification of $2$-near perfect numbers of the form $2^kp^m$ with $p$ an odd prime and $m\ge3$: every such number belongs to the family $2^k(2^{k+1}-1)^3$, where $p=2^{k+1}-1$ is a Mersenne prime and the omitted divisors are $p$ and $p^2$, and all such numbers are $2$-near perfect. In particular, there are infinitely many $2$-near perfect numbers in this family if and only if there are infinitely many Mersenne primes. Under the standard conjecture that there are infinitely many Mersenne primes, this disproves a conjecture of Aryan, Madhavani, Parikh, Slattery, and Zelinsky that only finitely many such numbers exist. Combined with prior work handling $m\in\{1,2\}$, this yields a complete characterization of all $2$-near perfect numbers with exactly two distinct prime factors.

\pagestyle{myheadings}
\markright{} 
\thispagestyle{empty}
\baselineskip=12.875pt
\vskip 30pt



\section{Introduction}

A classical object of study in number theory is the perfect number, an integer $n$ such that the sum of its divisors, $\sigma(n)$, equals $2n$. This concept has been generalized in various ways. One such generalization is pseudoperfect numbers. A number $n$ is said to be pseudoperfect if $n$ is equal to the sum of some subset of its proper positive divisors. 
Pollack and Shevelev \cite{PS} studied near-perfect numbers, which are sums of all but one of their proper positive divisors. Here we use the exact-$k$ convention, extending the definition of $2$-near perfect numbers in \cite{AMPSZ}: a number is $k$-near perfect if $\sigma(n) = 2n + \sum_{i=1}^k d_i$ for exactly $k$ distinct positive divisors $d_i$ of $n$. We call $d_1,\ldots,d_k$ the \textit{omitted divisors} of $n$. Under this convention the omitted divisors need not be proper; this matters in Section~\ref{sec:future}. A number $n$ is said to be \textit{abundant} if it satisfies $\sigma(n) > 2n$. If $n$ is $s$-near perfect for some $s>0$, then $n$ must be abundant, but the converse does not hold; 70 is a counterexample. Abundant numbers that are not pseudoperfect are said to be \textit{weird} and were investigated by Benkoski and Erd\H{o}s \cite{BE}. There is a classic conjecture that all weird numbers are even.

Recently, Jeba S, Roy, and Saikia \cite{JRS} introduced the related idea of a $k$-facile perfect number which is a number $n$ such that $\sigma(n)= 2n + \prod_{i=1}^k d_i$ where for $1 \leq i \leq k$ the $d_i$ are distinct divisors of $n$ satisfying $1<d_i<n$.

In the same paper where Pollack and Shevelev studied near-perfect numbers, they constructed the following three families of $1$-near perfect numbers:
\begin{enumerate}
\item $2^{t-1}(2^t-2^r-1)$, where $t\ge2$, $1\le r<t$, and
      $2^t-2^r-1$ is an odd prime. Here, $2^r$ is the omitted divisor.
\item $2^{2t-1}(2^t-1)$, where $t\ge2$ and $2^t-1$ is prime.
      Here, $2^t(2^t-1)$ is the omitted divisor.
\item $2^{t-1}(2^t-1)^2$, where $t\ge2$ and $2^t-1$ is prime.
      Here, $2^t-1$ is the omitted divisor.
\end{enumerate}

Ren and Chen \cite{RC} showed that all near perfect numbers with two distinct prime factors are either 40 or belong to one of the three families found by Pollack and Shevelev. Hasanalizade \cite{EH} gave a partial classification of near perfect numbers that are Fibonacci or Lucas numbers. Das and Saikia \cite{DasSaikia} gave constructions from known near-perfect numbers. Tang, Ma, and Feng \cite{TMF} showed that the only odd near perfect number with four distinct prime divisors is $173369889=(3^4)(7^2)(11^2)(19^2).$ Near-perfect and deficient-perfect numbers were also studied in Adajar's thesis \cite{Adajar}. Note there is some inconsistency in the literature as to whether a $k$-near perfect number can include $n$ as an omitted divisor. Pollack and Shevelev, as well as Ren and Chen do not so, whereas we include its removal as an option.

Cohen, Cordwell, Epstein, Kwan, Lott, and Miller \cite{CCEKLM}
established the asymptotic order of the counting function for each fixed
$k\ge4$ under the convention that \emph{at most} $k$ proper divisors
are omitted. Their convention differs from the exact-$k$ convention
used here, so that result is not an exact-$k$ classification.

This paper focuses on $2$-near perfect numbers, that is, numbers satisfying $\sigma(n) = 2n + d_1 + d_2$ where $d_1$ and $d_2$ are distinct positive divisors of $n$. Aryan, Madhavani, Parikh, Slattery, and the penultimate author of this paper \cite{AMPSZ} gave a complete classification of $2$-near perfect numbers of the form $2^k p^m$ with $p$ an odd prime and $m \in \{1,2\}$. Their results are the following two theorems.
\begin{theorem}\label{classification of 2 near perfect of form power of two times a prime}
Let $k\ge1$ and let $p$ be an odd prime. For a fixed unordered pair of
omitted divisors of a $2$-near perfect number $n=2^kp$, exactly one of
the following holds:
\begin{enumerate}
\item $p=2^k-1$, with omitted divisors $1,p$.
\item $p=2^{k+1}-2^a-2^b-1$, with omitted divisors $2^a,2^b$,
      where $1\le a<b\le k$.
\item $p=(2^{k+1}-2^a-1)/(1+2^b)$, with omitted divisors
      $2^a,2^bp$, where $1\le a,b\le k$.
\item $p=(2^{k+1}-1)/(1+2^a+2^b)$, with omitted divisors
      $2^ap,2^bp$, where $1\le a<b\le k$.
\end{enumerate}
Conversely, each case gives a $2$-near perfect number whenever the
specified value of $p$ is an odd prime.
\end{theorem}

\begin{proof}
Here $\sigma(n)-2n=2^{k+1}-1-p$ is even, so the omitted divisors
have the same parity. The only odd divisors of $n$ are $1$ and $p$;
using both gives $p=2^k-1$. Otherwise both divisors are even, and
the three possibilities are two powers of $2$, one power of $2$
and one divisor divisible by $p$, or two divisors divisible by $p$.
Substitution into $d_1+d_2=2^{k+1}-1-p$ gives the other three formulas.
The indicated bounds make the two divisors distinct divisors of $n$,
and substitution also proves each converse.
\end{proof}

\begin{theorem}\label{2 near of form power of two times prime squared}
Let $k\ge1$ and let $p$ be an odd prime. Then $n=2^kp^2$ is
$2$-near perfect if and only if $n\in\{18,36,200\}$.
\end{theorem}
\begin{proof}
The necessity is the classification in \cite{AMPSZ}.
For the converse, the respective omitted pairs are $(1,2)$,
$(1,18)$, and $(25,40)$, since
\[
\sigma(18)=39=36+1+2,\qquad
\sigma(36)=91=72+1+18,
\]
\[
\sigma(200)=465=400+25+40.
\]
\end{proof}

The same authors \cite{AMPSZ} conjectured finiteness for $m\ge2$. By Theorem~\ref{2 near of form power of two times prime squared}, this is equivalent to the following statement.

\begin{conjecture}There are only finitely many $2$-near perfect numbers of the form $2^k p^m$ for $m\geq 3$. \label{conj:finitely_many}
\end{conjecture}
Our main results disprove Conjecture \ref{conj:finitely_many}, subject to the widely believed conjecture that there are infinitely many Mersenne primes. We give a complete classification of $2$-near perfect numbers of the form $2^k p^m$ where $m \geq 3$.

\begin{theorem}\label{full characterization theorem for m>=3}
Let $k\ge1$, let $p$ be an odd prime, and let $m\ge3$. Then
$n=2^kp^m$ is $2$-near perfect if and only if $m=3$ and
$p=2^{k+1}-1$. In this case, $\{p,p^2\}$ is the unique unordered
pair of omitted divisors.
\end{theorem}

Note that Theorem \ref{full characterization theorem for m>=3} implies that there are infinitely many $2$-near perfect numbers of the form $n=2^kp^m$ with $m\ge3$ if and only if there are infinitely many Mersenne primes. Since it is widely conjectured that there are infinitely many Mersenne primes, the same should hold here.

A Lean 4 formalization of the main $2$-near perfect classification, using the Mathlib library, is available at \url{https://github.com/0xCUB3/Near-Perfect-Tester}.

We adopt the following conventions throughout. All numbers are positive integers unless otherwise specified. The letter $p$ always denotes an odd prime, $j$ is the unique integer satisfying $2^j<p<2^{j+1}$, and $v_2$ denotes the $2$-adic valuation. Exponents in a divisor $2^ap^b$ of $2^kp^m$ satisfy $0\le a\le k$ and $0\le b\le m$. Empty sums have value $0$.

The remainder of this paper is organized as follows. Section~\ref{sec:prelim} establishes foundational results, including properties of odd $2$-near perfect numbers and basic lemmas for numbers of the form $2^kp^m$. Section~\ref{sec:case1} considers Case I where $2^{k+1} \geq p^2+1$ and shows by contradiction that no $2$-near perfect numbers exist in this regime. Section~\ref{sec:case2} handles Case II where $2^{k+1} < p^2+1$, eliminating impossible divisor types and ultimately proving that $p = 2^{k+1}-1$, $m=3$, and the omitted divisors are $p$ and $p^2$. Section~\ref{sec:2deficient} extends our methods to $2$-deficient perfect numbers. Finally, Section~\ref{sec:future} presents computational results and discusses some open problems.

\section{Preliminaries}
\label{sec:prelim}

\subsection{Odd $2$-near perfect numbers}

The aforementioned results give a complete classification of $2$-near perfect numbers with two distinct prime factors. This is because there are no odd $2$-near perfect numbers with two distinct prime factors.

To prove this result, we recall the following basic properties of the $\sigma(n)$ function.  

\begin{lemma}
\label{sigma formula}
Let $n = p_1^{a_1} p_2^{a_2} \cdots p_k^{a_k}$, where $p_1, p_2, \ldots, p_k$ are distinct primes and $a_1, a_2, \ldots, a_k$ are positive integers. Then
\[ \sigma(n) = \prod_{i=1}^k \left(1 + p_i + p_i^2 + \cdots + p_i^{a_i}\right) = \prod_{i=1}^k \frac{p_i^{a_i+1} - 1}{p_i - 1}. \]
\end{lemma}





We will frequently use Lemma \ref{sigma formula} without explicit reference. We also require the following standard upper bound on $\sigma(n)$.

\begin{lemma}
\label{sigma upper bound}
    If $p$ is a prime number and $m$ is a positive integer, then $\frac{p}{p-1}>\frac{\sigma(p^m)}{p^m}$.
\end{lemma}
\begin{proof}
$$\frac{\sigma(p^m)}{p^m}=\frac{1+p+\ldots+p^m}{p^m}=\frac{p^{m+1}-1}{p^m(p-1)}<\frac{p^{m+1}}{p^m(p-1)}=\frac{p}{p-1}.$$

\end{proof}

\begin{corollary}
\label{cor:sigma_multiplicative_bound}
    If $n = p_1^{a_1} p_2^{a_2} \cdots p_k^{a_k}$ where $p_1, p_2, \ldots, p_k$ are distinct primes, then
    $$\frac{\sigma(n)}{n} < \prod_{i=1}^k \frac{p_i}{p_i-1}.$$
\end{corollary}
\begin{proof}
    By the multiplicativity of $\sigma$, by Lemma \ref{sigma formula} and Lemma \ref{sigma upper bound},
    $$\frac{\sigma(n)}{n} = \prod_{i=1}^k \frac{\sigma(p_i^{a_i})}{p_i^{a_i}} < \prod_{i=1}^k \frac{p_i}{p_i-1}.$$
\end{proof}

Note that the upper bound in Corollary~\ref{cor:sigma_multiplicative_bound} is, in a certain sense, the best possible bound. If one knows only the set of prime divisors of a number $n$, but not the corresponding exponents in its prime factorization, then no upper bound sharper than the one given can be obtained.

We now prove the main result of this subsection.

\begin{proposition}\label{prop:no_odd_2near}
    There are no odd $2$-near perfect numbers with exactly 2 prime factors.
\end{proposition}
\begin{proof}
    Suppose, for the sake of contradiction, that $n=p^mq^s$ is $2$-near perfect, where $p$ and $q$ are distinct odd primes and $m$ and $s$ are positive integers. By definition, $\sigma(n)=2n+d_1+d_2$. Therefore

    $$\frac{\sigma(n)}{n}>2$$ which implies that 
    \begin{equation*}
     \frac{15}{8} =  \frac{3}{3-1}\frac{5}{5-1} \geq   \frac{p}{p-1}\frac{q}{q-1}> \frac{\sigma(p^m)}{p^m}\frac{\sigma(q^s)}{q^s}> 2,
    \end{equation*}
    which is a contradiction. Note that at the third step above we are using that $\frac{x}{x-1}$ is a strictly decreasing function.
\end{proof}

Therefore, a $2$-near perfect number with exactly two distinct prime
factors must be even. For such a number $n=2^kp^m$,
Corollary~\ref{cor:sigma_multiplicative_bound} also gives
\[
\frac{\sigma(n)}n<\frac{2p}{p-1}\le3.
\]
Neither omitted divisor can equal $n$, since that would give
$\sigma(n)>3n$. Thus, although our definition allows $n$ itself as
an omitted divisor, it makes no difference to the two-prime
classification. We now treat the even case.


\subsection{Introductory Lemmas}

We begin by establishing several basic properties of even $2$-near perfect numbers with two prime factors.

\begin{lemma}\label{lem:basic_equation}
    Let $n=2^kp^m$ be a $2$-near perfect number with omitted divisors $d_1$ and $d_2$. Then
    \[(2^{k+1}-1)(1+p+\ldots+p^m)=2^{k+1}p^m+d_1+d_2.\]
\end{lemma}
\begin{proof}
    Apply Lemma \ref{sigma formula} with $n=2^kp^m$ and the definition $\sigma(n)=2n+d_1+d_2$.
\end{proof}

The equation in Lemma \ref{lem:basic_equation} can be manipulated to yield a more convenient form. By factoring out $p^m$ from the geometric series and rearranging terms, we obtain the following equivalent characterization.

\begin{lemma}\label{lem:equivalent_form}
    Let $n=2^kp^m$ be a $2$-near perfect number with omitted divisors $d_1$ and $d_2$. Then 
    \[(2^{k+1}-p)(1+p+\ldots+p^{m-1})=1+d_1+d_2.\]
\end{lemma}
\begin{proof}
   Starting from the equation in Lemma \ref{lem:basic_equation}, we factor the geometric series:
    \begin{align*}
    (2^{k+1}-1)(1+p+\ldots+p^m)&=2^{k+1}p^m+d_1+d_2\\
    (2^{k+1}-1)(1+p+\ldots+p^{m-1})+2^{k+1}p^m-p^m&=2^{k+1}p^m+d_1+d_2\\
    (2^{k+1}-1)(1+p+\ldots+p^{m-1})&=p^m+d_1+d_2\\
    (2^{k+1}-1)\frac{p^m-1}{p-1}&=p^m-1+1+d_1+d_2\\
    (2^{k+1}-1-(p-1))\frac{p^m-1}{p-1}&=1+d_1+d_2\\
    (2^{k+1}-p)(1+\ldots+p^{m-1})&=1+d_1+d_2.
    \end{align*}
\end{proof}

The parity structure of the omitted divisors depends on the parity of the exponent $m$, as the following result demonstrates.

\begin{lemma}\label{lem:parity_constraint}
    Let $n=2^kp^m$ be a $2$-near number perfect with omitted divisors $d_1$ and $d_2$. If $m$ is odd, then $d_1$ and $d_2$ have the same parity. If $m$ is even, then $d_1$ and $d_2$ have opposite parities.
\end{lemma}
\begin{proof}
    By Lemma \ref{lem:equivalent_form}, we have $(2^{k+1}-p)(1+p+\ldots+p^{m-1})=1+d_1+d_2$. Since $p$ is an odd prime and $2^{k+1}$ is even, the factor $(2^{k+1}-p)$ is odd.

    The sum $(1+p+\ldots+p^{m-1})$ consists of $m$ terms, each of which is odd (since $p$ is odd). Therefore, this sum is odd if $m$ is odd, and even if $m$ is even.

    Since $(2^{k+1}-p)$ is odd, the left-hand side $(2^{k+1}-p)(1+p+\ldots+p^{m-1})$ has the same parity as $(1+p+\ldots+p^{m-1})$. That is, the left-hand side is odd when $m$ is odd, and even when $m$ is even.

    Therefore, $1+d_1+d_2$ is odd when $m$ is odd and even when $m$ is even. This means $d_1+d_2$ is even when $m$ is odd and odd when $m$ is even.
\end{proof}

\begin{lemma}\label{lem:2k_geq_p}
    Let $n=2^kp^m$ be a $2$-near perfect number. Then $2^{k+1}-1\ge p$.
\end{lemma}
\begin{proof}
    By definition, $\sigma(n)=2n+d_1+d_2$, where $d_1,d_2$ are positive. We then have: 
    \begin{align*}
       \sigma(n)&>2n \\
       \frac{\sigma(2^k)\sigma(p^m)}{2^kp^m}&>2\\
       \frac{p}{p-1}\frac{\sigma(2^k)}{2^k}&>2,
    \end{align*} by Lemma \ref{sigma upper bound}. Therefore:
    \begin{align*}
       p(2^{k+1}-1)&>2^{k+1}(p-1) \\
       -p&>-2^{k+1}\\
       2^{k+1}&>p\\
    2^{k+1}-1&\ge p.
    \end{align*}
\end{proof}

\begin{lemma}\label{lem:k_geq_j}
    Let $n=2^kp^m$ be a $2$-near perfect number. Then $k\ge j$.
\end{lemma}
\begin{proof}
    By Lemma \ref{lem:2k_geq_p}, $2^{k+1}-1\ge p$. Therefore:
    \begin{align*}
        2^{k+1}&>2^j\\
        k+1&>j\\
        k&\ge j.
    \end{align*}
\end{proof}

\section{The Case $2^{k+1} \geq p^2+1$}
\label{sec:case1}

In this section, we show that no $2$-near perfect number of the form $2^kp^m$ with $m \geq 3$ can satisfy $2^{k+1} \geq p^2+1$. This is a key step in our classification.

Throughout this section, we adopt the following conventions: $d_1=2^ap^b$ denotes the greater omitted divisor, and $d_2=2^cp^d$ denotes the lesser omitted divisor, and $j$ denotes the unique integer satisfying $2^j < p < 2^{j+1}$. Our primary goal for this section is to establish the following constraint:

\begin{theorem}\label{thm:main_constraint}
Let $n=2^kp^m$ be a $2$-near perfect number with $p$ an odd prime and $m \geq 3$. Then $2^{k+1} < p^2+1$.
\end{theorem}

In this section, we assume $2^{k+1} \geq p^2+1$ and prove Theorem \ref{thm:main_constraint} by showing this leads to a contradiction. 

\subsection{Lower Bounds and Divisibility Constraints for the Exponent $m$}
\label{sec:constraints-exponent-m}

Under the assumption $2^{k+1} \geq p^2+1$, we first establish the structure of the omitted divisors.

\begin{lemma}\label{lem:divisor_structure_main}
Let $n=2^kp^m$ be a $2$-near perfect number with $m\geq 3$ and $2^{k+1}\geq p^2+1$. Then the greater omitted divisor is $n/2^j$, and the lesser omitted divisor is at least $\frac{n}{2^jp}$.
\end{lemma}

\begin{proof}
This follows from the contrapositive of Proposition \ref{prop:abundance_constraint} combined with Proposition \ref{prop:divisor_bounds}, both established below.
\end{proof}

\begin{proposition}\label{prop:abundance_constraint}
Let $n=2^kp^m$ be a $2$-near perfect number with $m\geq 3$. If $\frac{\sigma(n)}{n}<2+\frac{2}{p}$, then $2^{k+1}<p^2+1$.
\end{proposition}

\begin{proof}
Since $m\ge3$,
\[
\frac{\sigma(n)}n
=\frac{2^{k+1}-1}{2^k}\sum_{i=0}^{m}\frac1{p^i}
\ge\frac{(2^{k+1}-1)(p+1)(p^2+1)}{2^kp^3}.
\]
Combining this with $\sigma(n)/n<2(p+1)/p$ and canceling
positive factors gives
\[
(2^{k+1}-1)(p^2+1)<2^{k+1}p^2.
\]
Rearrangement yields $2^{k+1}<p^2+1$.
\end{proof}

\begin{proposition}\label{prop:divisor_bounds}
Let $n=2^kp^m$ be a $2$-near perfect number with $m\geq 3$. If $\frac{\sigma(n)}{n}\geq 2+\frac{2}{p}$, then the greater omitted divisor equals $\frac{n}{2^j}$ and the lesser omitted divisor is at least $\frac{n}{2^jp}$.
\end{proposition}

\begin{proof}
The hypothesis gives $d_1+d_2\ge2n/p$. Since $d_1>d_2$,
we have $d_1>n/p$. Thus the positive integer $q=n/d_1$ is a divisor of $n$ with $q<p$. Its only possible prime factor is $2$, so $q=2^r$ for an integer $0\le r\le j$.

By Corollary~\ref{cor:sigma_multiplicative_bound} and $p-1\ge2^j$,
\[
d_1<d_1+d_2=\sigma(n)-2n
<\frac{2n}{p-1}\le\frac{n}{2^{j-1}}.
\]
Hence $q>2^{j-1}$, forcing $r=j$ and $d_1=n/2^j$.
Finally,
\[
d_2\ge\frac{2n}{p}-\frac{n}{2^j}
=\frac{n(2^{j+1}-p)}{2^jp}\ge\frac{n}{2^jp},
\]
since $2^{j+1}-p$ is a positive integer.
\end{proof}

\subsection{Constraints on the Exponent $m$}
\label{sec:parity-and-m-odd}

We now derive several constraints on the exponent $m$.

\begin{proposition}\label{prop:d_bound}
Let $n=2^kp^m$ be a $2$-near perfect number with $m\geq 3$ and $2^{k+1}\geq p^2+1$. Then for $d_2 = 2^cp^d$,  $d \geq m-1$.
\end{proposition}

\begin{proof}
By Lemma \ref{lem:divisor_structure_main}, $d_2 \geq \frac{n}{2^jp} > \frac{n}{p^2} = 2^kp^{m-2}$. Suppose $d \leq m-2$. Then
\begin{align*}
    d_2 = 2^cp^d &> 2^kp^{m-2}\\
    2^cp^{m-2} &> 2^kp^{m-2}\\
    2^c &> 2^k\\
    c &> k
\end{align*}
but this contradicts the constraint that $c \leq k$ from the structure of the divisors of $n$.
\end{proof}

\begin{proposition}\label{prop:prime_power_divisibility}
Let $n=2^kp^m$ be a $2$-near perfect number with $m\geq 3$ and $2^{k+1}\geq p^2+1$. Then $p^{m-1}\mid 2^{k+1}-1$, and hence $p^{m-1} \leq 2^{k+1}-1$.
\end{proposition}

\begin{proof}
By Lemma~\ref{lem:divisor_structure_main}, $d_1=2^{k-j}p^m$,
and Proposition~\ref{prop:d_bound} gives $p^{m-1}\mid d_2$.
Thus $p^{m-1}$ divides every term on the right in
\[
(2^{k+1}-1)(1+p+\cdots+p^m)=2^{k+1}p^m+d_1+d_2.
\]
The geometric sum is congruent to $1$ modulo $p$, and hence is
coprime to $p^{m-1}$. Canceling it proves the divisibility. The inequality follows because $2^{k+1}-1$ is a positive multiple of $p^{m-1}$.
\end{proof}

\begin{proposition}\label{prop:m_odd}
Let $n=2^kp^m$ be a $2$-near perfect number with $m\geq 3$ and $2^{k+1}\geq p^2+1$. Then $m$ is odd.
\end{proposition}

\begin{proof}
Suppose $m$ is even, so $m\ge4$. Since
$2^{k+1}\ge p^2+1>2^{2j}$, we have $k\ge2j$.
Thus $d_1=2^{k-j}p^m$ is even. By Lemma~\ref{lem:parity_constraint},
$d_2$ is odd, so $d_2\le p^m$. Lemma~\ref{lem:divisor_structure_main}
now gives
\[
2^{k-j}p^{m-1}\le d_2\le p^m.
\]
Hence $2^{k-j}\le p<2^{j+1}$, which implies $k\le2j$.
Therefore $k=2j$. Proposition~\ref{prop:prime_power_divisibility} yields
\[
p^{m-1}\le2^{2j+1}-1<2p^2,
\]
whereas $m\ge4$ and $p\ge3$ give $p^{m-1}\ge p^3\ge3p^2$.
This contradiction proves that $m$ is odd.
\end{proof}

\subsection{$2$-Adic Valuations and Additional Restrictions}
\label{sec:2-adic-valuation}

We now examine how large powers of 2 divide various quantities related to $n$.

\begin{proposition}\label{prop:edge_case}
Let $n=2^kp^m$ be a $2$-near perfect number with $m\geq 3$ and $2^{k+1}\geq p^2+1$. Then $d_2 \neq 2^{k-j}p^{m-1}$.
\end{proposition}

\begin{proof}
Suppose, for the sake of contradiction, that $d_2 = 2^{k-j}p^{m-1}$. By Lemma \ref{lem:divisor_structure_main} $d_1 = 2^{k-j}p^m$. By Lemma \ref{lem:basic_equation}: 
\begin{align*}
    (2^{k+1}-1)(1+p+\ldots+p^m) &= 2^{k+1}p^m + 2^{k-j}p^m + 2^{k-j}p^{m-1}\\
    &= 2^{k-j}p^{m-1}(2^{j+1}p + p + 1).
\end{align*}

Because $m$ is odd by Proposition \ref{prop:m_odd}, we have $(1+p) | (1+p+\ldots+p^m)$. Since $\gcd(p, 1+p+\ldots+p^m) = 1$:
\begin{align*}
    (1+p) &| 2^{k-j}(2^{j+1}p + p + 1)\\
    (1+p) &| 2^{k+1}p.
\end{align*}

This implies $(1+p) | 2^{k+1}$, so $1+p$ is a power of 2. Since $2^j < p < 2^{j+1}$, we have $1+p = 2^{j+1}$.

Substituting back:
\begin{align*}
    (2^{k+1}-1)(1+p+\ldots+p^m) &= 2^{k-j}p^{m-1}(2^{2(j+1)})\\
    &= 2^{k+j+2}p^{m-1}.
\end{align*}

This gives $(2^{k+1}-1) | p^{m-1}$, so $2^{k+1}-1 \leq p^{m-1}$. Combined with Proposition \ref{prop:prime_power_divisibility}, we get $p^{m-1} = 2^{k+1}-1$.

Since $m$ is odd, $m-1$ is even, so $p^{m-1} \equiv 1 \pmod{4}$. But $2^{k+1}-1 \equiv 3 \pmod{4}$ for $k \geq 1$, giving a contradiction. 
\end{proof}

\begin{proposition}\label{prop:2adic_equation}
Let $n=2^kp^m$ be a $2$-near perfect number with $m\geq 3$ and $2^{k+1}\geq p^2+1$. Then $v_2(m+1)+v_2(p+1)-1=\min[a,c]$.
\end{proposition}

\begin{proof}
By Lemma~\ref{lem:divisor_structure_main}, $a=k-j$ and $b=m$.
Proposition~\ref{prop:d_bound} gives $d\in\{m-1,m\}$.
If $d=m$, then $a=c$ would make the omitted divisors equal.
If $d=m-1$, the same equality is excluded by
Proposition~\ref{prop:edge_case}. Thus $a\ne c$. Since $a,c\le k$,
exactly one term on the right below has the smallest valuation, and
\[
v_2\bigl(2^{k+1}p^m+2^ap^m+2^cp^d\bigr)=\min\{a,c\}.
\]

By Proposition~\ref{prop:m_odd}, $m$ is odd. Write $m+1=2^xy$,
where $x\ge1$ and $y$ is odd. Factoring gives
\[
\frac{p^{m+1}-1}{p-1}
=\left(\sum_{s=0}^{y-1}p^{s2^x}\right)
 (p+1)\prod_{r=1}^{x-1}(p^{2^r}+1).
\]
When $x=1$, the product is empty and has value $1$.
The sum in parentheses is odd. For $r\ge1$, the exponent $2^r$
is even, so $p^{2^r}\equiv1\pmod4$ and $v_2(p^{2^r}+1)=1$.
Therefore
\[
v_2(1+p+\cdots+p^m)
=v_2(p+1)+x-1=v_2(p+1)+v_2(m+1)-1.
\]
Comparing valuations in Lemma~\ref{lem:basic_equation}, with
$2^{k+1}-1$ odd, proves the claim.
\end{proof}

\subsection{Final steps to show $m=3$}
\label{sec:final-contradiction}

We now establish bounds that force $m=3$.

\begin{proposition}\label{prop:bound_d_m_minus_1}
Let $n=2^kp^m$ be a $2$-near perfect number with $m\geq 3$, $2^{k+1}\geq p^2+1$, and $d=m-1$. Then $p^{m-3} < 2(m+1)$.
\end{proposition}

\begin{proof}
Here $d_2=2^cp^{m-1}$. Lemma~\ref{lem:divisor_structure_main}
gives $c\ge k-j=a$, and Proposition~\ref{prop:edge_case} excludes
equality. Put $t=v_2(m+1)$. By Proposition~\ref{prop:2adic_equation},
\[
k-j=t+v_2(p+1)-1\le t+j,
\]
since $p+1\le2^{j+1}$. Thus $k\le t+2j$ and
\[
2^k\le2^{2j}2^t\le2^{2j}(m+1)<p^2(m+1).
\]
Proposition~\ref{prop:prime_power_divisibility} now gives
$p^{m-1}\le2^{k+1}-1<2p^2(m+1)$. Division by $p^2$ proves the result.
\end{proof}

\begin{proposition}
\label{prop:bound_d_m}
Let $n=2^kp^m$ be a $2$-near perfect number with $m\geq 3$, $2^{k+1}\geq p^2+1$, and $d=m$. Then $p^{m-3}<2(m+1)$.
\end{proposition}

\begin{proof} We first show that $a>c$.
Since $d_1=2^{k-j}p^m>d_2=2^cp^m$, we have $c<k-j$.
Put $t=v_2(m+1)$. Proposition~\ref{prop:2adic_equation} gives
\[
c=t+v_2(p+1)-1\le t+j.
\]
Also, Lemma~\ref{lem:divisor_structure_main} implies
$2^cp\ge2^{k-j}$. Since $p<2^{j+1}$, we obtain
$2^{c+j+1}>2^{k-j}$, hence $k\le c+2j\le t+3j$.
Consequently,
\[
2^{k+1}\le2^{3j+1}2^t\le2^{3j+1}(m+1)<2p^3(m+1).
\]
Both omitted divisors are divisible by $p^m$. The geometric sum in
Lemma~\ref{lem:basic_equation} is coprime to $p^m$, so
$p^m\mid2^{k+1}-1$. It follows that
$p^m\le2^{k+1}-1<2p^3(m+1)$, as required.
\end{proof}

\begin{proposition}\label{prop:not_p3_m5}
Let $n=2^kp^m$ be a $2$-near perfect number with $m\geq 3$ and $2^{k+1}\geq p^2+1$. Then it cannot be the case that both $m=5$ and $p=3$.
\end{proposition}

\begin{proof}
Suppose $p=3$ and $m=5$. Then $j=1$, and
$2^{k+1}\ge10$ implies $k\ge3$.
Lemma~\ref{lem:divisor_structure_main} gives
\[
d_1=2^{k-1}3^5,\qquad d_2=2^c3^d\ge2^{k-1}3^4.
\]
If $c\le k-3$, then
$d_2\le2^{k-3}3^5<2^{k-1}3^4$, a contradiction.
Thus $c\ge k-2$.

Since $1+3+\cdots+3^5=364$, Lemma~\ref{lem:basic_equation} gives
\[
(2^{k+1}-1)364=2^{k+1}3^5+d_1+d_2.
\]
Every term on the right is divisible by $2^{k-2}$. The valuation
of the left side is $v_2(364)=2$, so $k-2\le2$ and $k\le4$.
But Proposition~\ref{prop:prime_power_divisibility} would then give
$81\le2^{k+1}-1\le31$, a contradiction.
\end{proof}

\begin{proposition}\label{prop:m_equals_3}
Let $n=2^kp^m$ be a $2$-near perfect number with $m\geq 3$ and $2^{k+1}\geq p^2+1$. Then $m=3$.
\end{proposition}

\begin{proof}
By Proposition~\ref{prop:d_bound}, $d$ equals $m-1$ or $m$.
Propositions~\ref{prop:bound_d_m_minus_1} and~\ref{prop:bound_d_m}
therefore give $p^{m-3}<2(m+1)$.
For every integer $m\ge6$, however, $3^{m-3}>2(m+1)$.
Indeed, the base case is $27>14$, and the induction step follows from
\[
3^{(m+1)-3}=3\cdot3^{m-3}>6(m+1)>2(m+2).
\]
Thus $m\le5$. Proposition~\ref{prop:m_odd} leaves only $m=3$ or
$m=5$. The latter gives $p^2<12$, so $p=3$, which is excluded by
Proposition~\ref{prop:not_p3_m5}. Hence $m=3$.
\end{proof}

\subsection{Final Contradiction}

We conclude by showing that even $m=3$ leads to a contradiction.

\begin{proposition}\label{prop:d_equals_m}
Let $n=2^kp^m$ be a $2$-near perfect number with $m\geq 3$ and $2^{k+1}\geq p^2+1$. Then $d=m$.
\end{proposition}

\begin{proof}
Suppose $d\ne m$. By Propositions~\ref{prop:d_bound}
and~\ref{prop:m_equals_3}, we have $m=3$ and $d=2$.
Lemma~\ref{lem:divisor_structure_main} gives
$d_1=2^ap^3$, $d_2=2^cp^2$, and $a=k-j$ with $c\ge a$.
Proposition~\ref{prop:edge_case} excludes $c=a$, so $c>a$.
By Proposition~\ref{prop:2adic_equation},
\[
a=v_2(p+1)+1,\qquad 2^a\le2(p+1).
\]
Proposition~\ref{prop:prime_power_divisibility} allows us to write
$2^{k+1}-1=xp^2$ with $x$ a positive integer. Since $2^j<p$,
\[
2^{k+1}=2^a2^{j+1}<4p(p+1),
\qquad
x<\frac{4(p+1)}p\le\frac{16}{3}<7.
\]
On the other hand, $2^{k+1}\ge p^2+1\ge10$ implies $k\ge3$.
Reduction modulo $8$, using $p^2\equiv1\pmod8$, gives
$x\equiv7\pmod8$. A positive integer with that residue is at
least $7$, contradicting the upper bound.
\end{proof}

\begin{proposition}\label{prop:final_contradiction}
There exists no $2$-near perfect number $n=2^kp^3$ with $2^{k+1}\geq p^2+1$.
\end{proposition}

\begin{proof}
Suppose such a number exists. By Proposition~\ref{prop:d_equals_m}
and Lemma~\ref{lem:divisor_structure_main}, its omitted divisors are
$d_1=2^{k-j}p^3$ and $d_2=2^cp^3$, with $c<k-j$.
Proposition~\ref{prop:2adic_equation} gives
\[
c=v_2(p+1)+1,\qquad 2^c\le2(p+1).
\]
Both omitted divisors are divisible by $p^3$, so
Lemma~\ref{lem:basic_equation} gives $p^3\mid2^{k+1}-1$.
Write $2^{k+1}-1=xp^3$ with $x\ge1$.
Dividing that lemma's equation by $p^3$ and substituting for
$2^{k+1}$ gives
\[
x(1+p+p^2)=1+2^{k-j}+2^c.
\]
Since $2^{j+1}\ge p+1$ and $x\ge1$,
\[
2^{k-j}=\frac{xp^3+1}{2^{j+1}}
\le\frac{x(p^3+1)}{p+1}=x(p^2-p+1).
\]
Combining these bounds yields
\[
x(1+p+p^2)\le1+x(p^2-p+1)+2(p+1),
\quad\text{so}\quad 2xp\le2p+3.
\]
As $p\ge3$, this forces $x=1$. Thus
\[
2^{k+1}=p^3+1=(p+1)(p^2-p+1).
\]
The second factor is odd and greater than $1$, and therefore cannot
divide a power of $2$. This is the required contradiction.
\end{proof}

\begin{proof}[Proof of Theorem \ref{thm:main_constraint}]
Suppose $n=2^kp^m$ is $2$-near perfect with $m \geq 3$ and $2^{k+1} \geq p^2+1$. By Proposition \ref{prop:m_equals_3}, we must have $m = 3$. But then Proposition \ref{prop:final_contradiction} shows that no such $2$-near perfect number with $m = 3$ can exist under the assumption $2^{k+1} \geq p^2+1$. Therefore, we must have $2^{k+1} < p^2+1$.
\end{proof}

\section{The Case $2^{k+1} < p^2+1$}
\label{sec:case2}

\subsection{Elimination of Impossible Omitted Divisor Types}
\label{sec:no-impossible-omitted}

\begin{proposition}\label{prop:p_divisibility_constraint}
   If $n=2^kp^m$ is $2$-near perfect, where $m\ge3$ and $2^{k+1}<p^2+1$, then $p^{m-2}$ divides one omitted divisor, and $p^2$ does not divide the other. 
\end{proposition}

\begin{proof}
    Let $n=2^kp^m$ be a $2$-near perfect number with $m\ge3$ and $2^{k+1}<p^2+1$. Let the omitted divisors be $d_1=2^ap^b$ and $d_2=2^cp^d$, and let $d\ge b$.

    By Lemma \ref{lem:basic_equation},
    \begin{align*}
        (2^{k+1}-1)(p^m+p^{m-1}+\dots+1)&=2^{k+1}p^{m}+d_1+d_2.
    \end{align*}

    Since $m\ge3$, $p^2\mid2^{k+1}p^m$.
    
    Suppose, for the sake of contradiction, that $p^2$ divides both $d_1$ and $d_2$. Then:
    \begin{align*}
        p^2&\mid2^{k+1}p^{m}+d_1+d_2\\
        p^2&\mid(2^{k+1}-1)(p^m+p^{m-1}+ \cdots +1).
    \end{align*}

    Since $(p^2,p^m+p^{m-1}+\dots+1)=1$, we conclude $p^2\mid2^{k+1}-1$. Hence, $p^2\le2^{k+1}-1$. But $2^{k+1}-1<p^2$, by construction. Thus, the supposition that $p^2$ divides $d_1$ and $d_2$ is false.

    By Lemma \ref{lem:equivalent_form} we have 
    \[1+p+p^2+\cdots+p^{m-1}\mid1+2^ap^b+2^cp^d,\]

    which implies that
    \[1+p+p^2+\cdots+p^{m-1}\le1+2^ap^b+2^cp^d.\]
    
    Since $b \leq d$ and $2^a,2^c \leq 2^k$:
    \[p+p^2+\cdots+p^{m-1}\le2^kp^b+2^kp^d \leq 2^k(p^b+p^d) \leq 2^{k+1}p^d.\]
    
    By our hypothesis $2^{k+1}<p^2+1$, we have $2^{k+1} \leq p^2$, so:
    \[p+p^2+\cdots+p^{m-1}\le p^2 \cdot p^d = p^{d+2}.\]
    
    In particular, looking at the two highest-degree terms on the left:
    \[p^{m-2}+p^{m-1}\le p+p^2+\cdots+p^{m-1}\le p^{d+2}.\]

    Now, suppose, for the sake of contradiction, that $d \leq m-3$. Then $d+2 \leq m-1$, so
    \[p^{m-2}+p^{m-1}\le p^{d+2} \leq p^{m-1},\]
    which gives $p^{m-2} \leq 0$, a contradiction. Therefore, $d \geq m-2$.
\end{proof}

\begin{proposition}\label{prop:no_pm_divisor}
    If $n=2^kp^m$ is $2$-near perfect, and $m\ge3$ and $2^{k+1}<p^2+1$, then $2^cp^m$ is not an omitted divisor.
\end{proposition}

\begin{proof}
Suppose the omitted divisors are $2^ap^b$ and $2^cp^m$.
Proposition~\ref{prop:p_divisibility_constraint} gives $b\in\{0,1\}$.
Let $S=1+p+\cdots+p^{m-1}$. Lemma~\ref{lem:equivalent_form}
and $p^m\equiv1\pmod S$ give
\[
S\mid1+2^ap^b+2^c,
\qquad S\le1+2^ap^b+2^c.
\]
If $b=0$, then $p+p^2\le S-1\le2^a+2^c\le2^{k+1}<p^2+1$,
which is impossible. Thus $b=1$.
If $m\ge4$, then
\[
p+p^2+p^3\le S-1\le2^k(p+1)
<\frac{p^3+p^2+p+1}{2}<p+p^2+p^3,
\]
also impossible. Therefore $m=3$. Lemma~\ref{lem:basic_equation} now reads
\[
(2^{k+1}-1)(p+1)(p^2+1)
=p\bigl((2^{k+1}+2^c)p^2+2^a\bigr).
\]
Since $\gcd(p,p^2+1)=1$, reduction modulo $p^2+1$ gives
\[
p^2+1\mid D,\qquad D=2^{k+1}+2^c-2^a>0.
\]
If $c\le a$, then $D\le2^{k+1}<p^2+1$, a contradiction.
Hence $c>a$. Since $m$ is odd, Lemma~\ref{lem:parity_constraint}
forces both omitted divisors to be even; thus $a\ge1$ and $c\ge2$.
Moreover,
\[
D<2^{k+1}+2^k<\frac32(p^2+1).
\]
Being a positive multiple of $p^2+1$, $D$ must equal $p^2+1$.
Modulo $4$, this equality gives $2\equiv-2^a$, hence $a=1$.
But then the left side of Lemma~\ref{lem:basic_equation} is divisible
by $4$, since $(p+1)(p^2+1)$ is, whereas its right side is
\[
2^{k+1}p^3+2p+2^cp^3\equiv2\pmod4.
\]
This contradiction excludes the proposed divisor.
\end{proof}

\begin{proposition}\label{prop:no_power_of_2_divisor}
    If $n=2^kp^m$ is a $2$-near perfect number with $m \ge 3$ and $2^{k+1} < p^2+1$, then neither omitted divisor can be a power of 2.
\end{proposition}

\begin{proof}
Suppose the omitted divisors are $d_1=2^a$ and $d_2=2^cp^d$.
Propositions~\ref{prop:p_divisibility_constraint} and~\ref{prop:no_pm_divisor}
give $m-2\le d\le m-1$, so in particular $d\ge1$.
If $p\le2^k$, Lemma~\ref{lem:equivalent_form} gives
\[
2^k(1+p+\cdots+p^{m-1})
\le1+2^a+2^cp^d\le1+2^k+2^kp^{m-1}.
\]
Subtracting the constant and last terms yields
$2^k\sum_{i=1}^{m-2}p^i\le1$, impossible for $m\ge3$.
Thus $p>2^k$.

Reduction of Lemma~\ref{lem:basic_equation} modulo $p$ gives
$p\mid2^{k+1}-1-2^a$. This integer is positive, because
$2^{k+1}-1-2^a\ge2^k-1\ge1$, and is less than $2p$.
It must therefore equal $p$. Thus
\[
p=2^{k+1}-1-2^a.
\]
As $p$ is odd, $a\ge1$. Substitution in
Lemma~\ref{lem:equivalent_form}, followed by subtraction and division
by $p$, gives
\[
(1+2^a)(1+p+\cdots+p^{m-2})=2^cp^{d-1}.
\]
But $1+2^a$ is odd and greater than $1$, and
\[
\gcd(1+2^a,p)=\gcd(1+2^a,2^{k+1})=1.
\]
Since it is coprime to $p$, $1+2^a$ must divide $2^c$, which is
impossible because it is odd and greater than $1$.
\end{proof}

\subsection{Forcing the Form: $p = 2^{k+1} - 1$ and $m = 3$}
\label{sec:force-p-and-m}

\begin{proposition}\label{prop:p_is_mersenne_form}
    If $n=2^kp^m$ is a $2$-near perfect number with $m \ge 3$ and $2^{k+1} < p^2+1$, then $p$ must be of the form $2^{k+1}-1$.
\end{proposition}

\begin{proof}
Order the omitted divisors as $2^ap^b,2^cp^d$ with $b\le d$.
Propositions~\ref{prop:p_divisibility_constraint},
\ref{prop:no_power_of_2_divisor}, and~\ref{prop:no_pm_divisor}
give $b=1$ and $1\le d\le m-1$.
Thus reduction of Lemma~\ref{lem:basic_equation} modulo $p$ gives
$p\mid2^{k+1}-1$.

If $p\le2^k$, Lemma~\ref{lem:equivalent_form} would imply
\[
2^k(1+p+\cdots+p^{m-1})
\le1+2^ap+2^cp^d\le1+2^kp+2^kp^{m-1}.
\]
Cancelling the $p$ and $p^{m-1}$ terms gives
\[
2^k\left(1+\sum_{i=2}^{m-2}p^i\right)\le1.
\]
The sum is empty when $m=3$, but the constant term alone makes
this impossible because $2^k\ge2$. Therefore $p>2^k$.
The positive multiple $2^{k+1}-1$ of $p$ is less than $2p$,
so it equals $p$.
\end{proof}

\begin{proposition}\label{prop:no_p_and_pm2_divisors}
    If $n=2^kp^m$ is a $2$-near perfect number with $m \ge 3$ and $2^{k+1} < p^2+1$, then the omitted divisors cannot be of the forms $2^ap$ and $2^cp^{m-2}$.
\end{proposition}

\begin{proof}
Suppose the omitted divisors are $2^ap$ and $2^cp^{m-2}$.
By Proposition~\ref{prop:p_is_mersenne_form}, $p=2^{k+1}-1$.
Lemma~\ref{lem:equivalent_form}, after subtraction of $1$ and division
by $p$, gives
\[
\frac{p^{m-1}-1}{p-1}
=2^a+2^cp^{m-3}
\le\frac{p+1}{2}(1+p^{m-3}).
\]
Multiplying by $2(p-1)$ and rearranging yields
\[
(p^2+1)(p^{m-3}-1)\le0.
\]
Since $m\ge3$ and $p>1$, this forces $m=3$.
The original equality then becomes
\[
2^{k+1}=p+1=2^a+2^c.
\]
Both summands are at most $2^k$, so equality forces $a=c=k$.
The two omitted divisors would then be equal, a contradiction.
\end{proof}

\subsection{Final Classification for $2^{k+1} < p^2 + 1$}
\label{sec:final-characterization}

\begin{proposition}\label{prop:case2_classification}
If $n=2^kp^m$ is $2$-near perfect with $m\ge3$ and
$2^{k+1}<p^2+1$, then $m=3$, $p=2^{k+1}-1$, and the unique
unordered omitted pair is $\{p,p^2\}$.
\end{proposition}
\begin{proof}
By Proposition~\ref{prop:p_is_mersenne_form}, $p=2^{k+1}-1$.
Write the omitted divisors as $2^ap^b,2^cp^d$ with $b\le d$.
Propositions~\ref{prop:p_divisibility_constraint}
and~\ref{prop:no_power_of_2_divisor} give $b=1$ and $d\ge m-2$.
Proposition~\ref{prop:no_pm_divisor} gives $d\le m-1$, and
Proposition~\ref{prop:no_p_and_pm2_divisors} excludes $d=m-2$.
Hence $d=m-1$.

Lemma~\ref{lem:equivalent_form}, after subtraction of $1$ and
division by $p$, now gives
\[
1+p+\cdots+p^{m-2}=2^a+2^cp^{m-2}.
\]
Modulo $p$, this says $2^a\equiv1$. Since
$1\le2^a\le2^k<p$, we obtain $a=0$.
Subtracting $1$ and dividing by $p$ again gives
\[
1+p+\cdots+p^{m-3}=2^cp^{m-3}.
\]
If $m>3$, the two sides are respectively $1$ and $0$ modulo $p$,
a contradiction. Thus $m=3$ and the equality becomes $1=2^c$,
so $c=0$. The argument applies to any omitted pair, proving uniqueness.
\end{proof}


\begin{proof}[Proof of Theorem~\ref{full characterization theorem for m>=3}]
The necessity and uniqueness follow from Theorem~\ref{thm:main_constraint}
and Proposition~\ref{prop:case2_classification}.
Conversely, if $p=2^{k+1}-1$ is prime and $m=3$, then
\begin{align*}
\sigma(2^kp^3)
&=p(1+p+p^2+p^3)\\
&=(p+1)p^3+p+p^2
=2(2^kp^3)+p+p^2.
\end{align*}
The numbers $p$ and $p^2$ are distinct positive divisors of $2^kp^3$,
so this is a $2$-near perfect number.
\end{proof}

\section{$2$-Deficient Perfect Numbers}
\label{sec:2deficient}


\begin{definition}
    A positive integer $n$ is called \emph{$2$-deficient perfect} if there exist two distinct positive divisors $d_1$ and $d_2$ of $n$ such that:
    \[\sigma(n) = 2n - d_1 - d_2,\]
    where $\sigma(n)$ denotes the sum of all positive divisors of $n$.
\end{definition}

Note that prior work by Chen \cite{Chen2019} classified all odd $2$-deficient perfect numbers with exactly two distinct prime factors. 

The algebra used in Lemma~\ref{lem:equivalent_form} gives the identity
\begin{equation}\label{eq:deficient_identity}
d_1+d_2=1+(p-2^{k+1})(1+p+\dots+p^{m-1}).
\end{equation}

\begin{proposition}\label{prop:2def_prelim}
    If $n=2^kp^m$ is a $2$-deficient perfect number, then
    \begin{enumerate}
        \item[(a)] The prime $p$ must satisfy $p > 2^{k+1}$.
        \item[(b)] At least one of the omitted divisors, $d_1$ or $d_2$, must be a power of 2.
    \end{enumerate}
\end{proposition}

\begin{proof}
    (a) Rearranging equation~\eqref{eq:deficient_identity}, we have $(2^{k+1}-p)(1+p+\dots+p^{m-1}) = 1-d_1-d_2$. Since $d_1$ and $d_2$ are distinct positive divisors, $1-d_1-d_2$ is a negative integer. The sum $(1+p+\dots+p^{m-1})$ is positive. Therefore, the term $(2^{k+1}-p)$ must be negative. As it is an integer, $2^{k+1}-p \le -1$, which implies $p \ge 2^{k+1}+1$. Thus, $p > 2^{k+1}$.

    (b) Suppose, for the sake of contradiction, that $p$ divides both $d_1$ and $d_2$. Then $p$ also divides their sum, $d_1+d_2$. From the definition $\sigma(n) = 2n - (d_1+d_2)$, it follows that $d_1+d_2 = 2n - \sigma(n)$. We have:
    \[ d_1+d_2 = 2 \cdot 2^k p^m - \sigma(2^k)\sigma(p^m) = 2^{k+1}p^m - (2^{k+1}-1)(1+p+\dots+p^m). \]
    Taking this modulo $p$, we get $d_1+d_2 \equiv -(2^{k+1}-1)(1) \pmod{p}$. Since $p \mid d_1+d_2$, it must be that $p \mid -(2^{k+1}-1)$, which means $p \le 2^{k+1}-1$. This contradicts part (a), which showed $p > 2^{k+1}$. Thus, the assumption is false, and at least one divisor is not divisible by $p$, meaning it must be a power of 2.
\end{proof}

\begin{theorem}\label{thm:2def_classification}
    An integer $n=2^kp^m$ with $p$ an odd prime and $m \neq 2$ is $2$-deficient perfect if and only if $m=1$ and $p$ is a prime of the form $p=2^{k+1}+2^a+2^b-1$ for integers $1\le a<b\le k$.
\end{theorem}

\begin{remark}
    The hypothesis $m \neq 2$ is necessary: $n = 2 \cdot 5^2 = 50$ is $2$-deficient perfect with $m = 2$, since $\sigma(50) + 2 + 5 = 100 = 2 \cdot 50$. The $m=2$ case is not treated here.
\end{remark}

\begin{proof}
Suppose $n$ is $2$-deficient perfect with deficient divisors $d_1,d_2$.
Put $Q=2^{k+1}$ and $h=p-Q$. By Proposition~\ref{prop:2def_prelim},
$h>0$. Equation~\eqref{eq:deficient_identity} becomes
\[
d_1+d_2=(h+1)+h\sum_{b=1}^{m-1}p^b.
\]
Since $0<h<h+1<p$, this is the base-$p$ expansion of the deficit,
with exactly $m$ nonzero digits.

Write $d_i=2^{a_i}p^{b_i}$, where $0\le a_i\le k$ and
$0\le b_i\le m$. If $b_1\ne b_2$, the sum $d_1+d_2$ has exactly
two nonzero base-$p$ digits, because each coefficient is at most
$2^k<p$. If $b_1=b_2$, distinctness gives $a_1\ne a_2$, and
\[
2^{a_1}+2^{a_2}\le2^k+2^{k-1}<Q<p.
\]
There is again no carry, and the sum has one nonzero digit.
Uniqueness of base-$p$ expansion therefore gives $m\le2$.
The hypothesis $m\ne2$ now forces $m=1$.

For $m=1$, the deficit is $d_1+d_2=h+1<p$. Neither deficient
divisor can be divisible by $p$, so they are distinct powers
$2^a,2^b$ with $0\le a<b\le k$. Hence
\[
p=Q+h=2^{k+1}+2^a+2^b-1.
\]
Because $h+1$ is even and $a\ne b$, neither exponent can be zero.
Thus $1\le a<b\le k$.

Conversely, suppose $m=1$ and $p=2^{k+1}+2^a+2^b-1$ is prime
with $1\le a<b\le k$. Direct calculation gives
\[
2n-\sigma(n)
=2^{k+1}p-(2^{k+1}-1)(p+1)
=1+p-2^{k+1}=2^a+2^b.
\]
The two powers of $2$ are distinct proper divisors of $n$,
so $n$ is $2$-deficient perfect.
\end{proof}

\section{Future Work}
\label{sec:future}

The complete characterization of $2$-near perfect numbers with two distinct prime factors opens several directions for further investigation.

\subsection{Computational Results}


Our computational search was implemented in C++; the source code is available at \url{https://github.com/0xCUB3/Near-Perfect-Tester}.

\subsection{Multi-Prime $2$-near Perfect Numbers}
The smallest examples are shown in the table below:

\begin{table}[htbp]
\begin{center}
\begin{tabular}{|c|c|c|}
\hline
Number & $\sigma$ & Omitted Divisors \\
\hline
30 & 72 & 2, 10 \\
66 & 144 & 1, 11 \\
84 & 224 & 14, 42 \\
90 & 234 & 9, 45 \\
126 & 312 & 18, 42 \\
132 & 336 & 6, 66 \\
156 & 392 & 2, 78 \\
180 & 546 & 6, 180 \\
\hline
\end{tabular}
\\[2mm]
\small{The eight smallest $2$-near perfect numbers with three distinct prime factors}
\end{center}
\end{table}

The omitted divisors need not contain the largest prime factor,
even when all possible omitted pairs are considered. For example,
\[
3250=2\cdot5^3\cdot13,\qquad
\sigma(3250)=3\cdot156\cdot14=6552=2\cdot3250+2+50.
\]
Neither $2$ nor $50$ is divisible by $13$. This is the only omitted
pair: the smaller member of a pair of distinct positive divisors
summing to $52$ must be a divisor less than $26$. These divisors are
$1,2,5,10,13,25$, and only the complement $50$ of $2$ is also a
divisor of $3250$. Classifying families with three distinct prime
factors remains a natural next problem.


\subsection{$3$-Near Perfect Numbers}
The smallest $3$-near perfect numbers are shown in the table below:

\begin{table}[htbp]
\begin{center}
\begin{tabular}{|c|c|c|}
\hline
Number & $\sigma$ & Omitted Divisors \\
\hline
24 & 60 & 1, 3, 8 or 2, 4, 6 \\
30 & 72 & 1, 5, 6 \\
36 & 91 & 1, 6, 12 or 3, 4, 12 or 4, 6, 9 \\
40 & 90 & 1, 4, 5 \\
42 & 96 & 2, 3, 7 \\
48 & 124 & 1, 3, 24 or 4, 8, 16 \\
54 & 120 & 1, 2, 9 \\
60 & 168 & 3, 15, 30 or 6, 12, 30 \\
\hline
\end{tabular}
\\[2mm]
\small{The eight smallest $3$-near perfect numbers}
\end{center}
\end{table}

Many $3$-near perfect numbers have multiple valid combinations of omitted divisors, making it more difficult to identify analogous families.

\subsection{Hybrid Forms}
Beyond the standard definition, we identified several variants for future research:

\begin{enumerate}
\item \textbf{Hybrid $2$-near Perfect Numbers:} Numbers where one divisor is added and one is subtracted, satisfying $\sigma(n) = 2n + d_1 - d_2$.

    \item \textbf{2-Deficient Perfect Numbers:} Numbers where both divisors are subtracted, satisfying $\sigma(n) = 2n - d_1 - d_2$.
\end{enumerate}

\vskip 20pt\noindent {\bf Acknowledgement}

The research for this paper was performed during the Hopkins School 2025 Spring Mathematics Seminar. 

\bibliographystyle{plainnat}
\bibliography{references}

\end{document}